\documentclass[12pt]{amsart}
\usepackage[T1]{fontenc}
\usepackage{dsfont}
\usepackage{mathrsfs}
\usepackage[a4paper,asymmetric]{geometry}
\usepackage{mathscinet}
\usepackage{fullpage}
\usepackage{latexsym}
\usepackage{amsthm}
\usepackage{amssymb}
\usepackage{amsfonts}
\usepackage{amsmath}

\newtheorem{theorem}{Theorem}[section]
\newtheorem{thm}[theorem]{Theorem}

\newtheorem{lem}[theorem]{Lemma}
\newtheorem{proposition}[theorem]{Proposition}
\newtheorem{prop}[theorem]{Proposition}
\newtheorem{corollary}[theorem]{Corollary}

\theoremstyle{definition}

\newtheorem{defn}[theorem]{Definition}

\theoremstyle{remark}

\newtheorem{rem}[theorem]{Remark}
\numberwithin{equation}{section}

 \DeclareMathAlphabet{\mathpzc}{OT1}{pzc}{m}{it}
 \newcommand{\To}{\longrightarrow}

 \newcommand{\E}{\mathbb{E}}            
 \newcommand{\T}{\mathbb{T}}
 \newcommand{\e}{\varepsilon}
 \newcommand{\p}{\partial}

 \newcommand{\Ll}{\langle}
 \newcommand{\Rr}{\rangle}

 \newcommand{\N}{\mathbb{N}}
 \newcommand{\R}{\mathbb{R}}
 \newcommand{\Z}{\mathbb{Z}}
 \newcommand{\PP}{\mathbb{P}}
 \newcommand{\mcl}{\mathcal}
 
 \newcommand{\Be}{\begin{equation}}
 \newcommand{\Ee}{\end{equation}}
 \newcommand{\Bs}{\begin{split}}
 \newcommand{\Es}{\end{split}}
  \newcommand{\Bes}{\begin{equation*}}
 \newcommand{\Ees}{\end{equation*}}
 \newcommand{\BT}{\begin{thm}}
 \newcommand{\ET}{\end{thm}}
 \newcommand{\Bp}{\begin{proof}}
 \newcommand{\Ep}{\end{proof}}
 \newcommand{\BL}{\begin{lem}}
 \newcommand{\EL}{\end{lem}}
 \newcommand{\BP}{\begin{proposition}}
 \newcommand{\EP}{\end{proposition}}
 \newcommand{\BC}{\begin{corollary}}
 \newcommand{\EC}{\end{corollary}}
 \newcommand{\BR}{\begin{rem}}
 \newcommand{\ER}{\end{rem}}
 \newcommand{\BD}{\begin{defn}}
 \newcommand{\ED}{\end{defn}}
 \newcommand{\BI}{\begin{itemize}}
 \newcommand{\EI}{\end{itemize}}
 \newcommand{\eqn}{equation}
 \newcommand{\tl}{\tilde}
 \newcommand{\lf}{\left}
 \newcommand{\rg}{\right}

  \newcommand{\lt}{\left}
 \newcommand{\rt}{\right}

\begin{document}
\title
[Ergodicity of stochastic Ginzburg-Landau equation with $\alpha$-stable noises]{Ergodicity of stochastic real Ginzburg-Landau equation driven by $\alpha$-stable noises}

\author[L. Xu]{Lihu Xu}
\address{Department of Mathematics, Brunel University,
Kingston Lane,
Uxbridge,
Middlesex UB8 3PH, United Kingdom}
\email{Lihu.Xu@brunel.ac.uk}

\maketitle

\begin{abstract} \label{abstract}
We study the ergodicity of stochastic real Ginzburg-Landau equation driven by additive $\alpha$-stable noises,
showing that as $\alpha \in (3/2,2)$, this stochastic system admits a unique invariant measure.
After establishing the existence of invariant measures by the same method as in
\cite{DXZ11},
we prove that the system is strong Feller and accessible to zero.
These two properties imply the ergodicity
by a simple but useful criterion in \cite{Hai09}.
To establish the strong Feller property, we need to truncate the nonlinearity and apply a gradient estimate established in
\cite{PZ11} (or see \cite{PSXZ11} for a general version for the finite dimension systems). Because
the solution has discontinuous trajectories and the nonlinearity is not Lipschitz,
we can not solve a control problem to get irreducibility. Alternatively,
we use a replacement, i.e., the fact that the system is accessible to zero.
In section \ref{s:ConSta},
we establish a maximal inequality for stochastic
$\alpha$-stable convolution, which is crucial for studying
the well-posedness, strong Feller property and the accessibility of the mild
solution. We hope this inequality will also be useful for studying other
SPDEs forced by $\alpha$-stable noises.
 \\

\end{abstract}

\section{Introduction}
We shall study the ergodicity of
stochatic real Ginzburg-Landau equation driven by $\alpha$-stable noises
on torus $\T=\R/\Z$ as follows:
\Be \label{e1.1}
d X-\p^2_\xi Xdt-(X-X^3)dt=dL_t
\Ee
where $X:[0, \infty): \R^{+} \times \T \rightarrow \R$ and $L_t$ is some cylindrical $\alpha$-stable noises.
The more details about Eq. \eqref{e1.1} will be given in the next section. \\

For the study of invariant measures
and the long-time behaviour of stochastic systems driven by $\alpha$-stable type noises,
there seem
only several results (cf.~\cite{XZ09, XZ10, PXZ10, PSXZ11, Ku09, X11}).
\cite{XZ09, XZ10} studied the exponential mixing of stochastic spin systems with
white $\alpha$-stable noises, while \cite{PXZ10, PSXZ11} obtained exponential mixing
for a family of semi-linear SPDEs with Lipschitz nonlinearity. \cite{Ku09} obtained a nice
criterion for the exponential mixing of a family of SDEs forced by L\'evy noises, it covers 1D SDEs driven by
$\alpha$-stable noises. \cite{X11} proved the exponential mixing for a family of 2D SDEs forced by
degenerate $\alpha$-stable noises with $0<\alpha<2$. \cite{DXZ11} obtained the existence of invariant measures
for 2D stochastic Navier-Stokes
equations forced by $\alpha$-stable noises with $\alpha \in (1,2)$. \\

In this paper, we shall study the ergodicity of stochastic real Ginzburg-Landau equation driven by additive $\alpha$-stable noises,
showing that as $\alpha \in (3/2,2)$, the system \eqref{e1.1} admits a unique invariant measure.
After establishing the existence of invariant measures by the same method as in
\cite{DXZ11},
we prove that the system is strong Feller and accessible to zero. These two properties imply the ergodicity
by a simple but useful criterion in \cite{Hai09}.
To establish the strong Feller property, we need to truncate the nonlinearity and apply a gradient estimate established in
\cite{PZ11} (or see \cite{PSXZ11} for a general version for the finite dimension systems). Due to
the non-Lipschitz nonlinearity and the discontinuous trajectories, unlike the case of SPDEs forced by
Wiener noises,
we can not solve a control problem to get irreducibility. Alternatively,
we use a replacement, i.e., the fact that the system is accessible to zero.
 \\

SPDEs with L\'evy noises have been intensively studied in recent years (\cite{AWZ98,AMR09,Ok08,MaRo09,PZ11,XZ09,WX10,TW05,PeZa07,DX07, Do08, DXi09}), but most of them assume that the noises are square integrable. This
restriction rules out the interesting $\alpha$-stable noises. The loss of the second moment of $\alpha$-stable noises makes
many nice analysis tools, such as Burkholder-Davis-Gundy inequality and Da Prato-Kwapie\'n-Zabczyk's factorization technique (\cite{DPZ92}), not available. Consequently, some important estimates such as maximal inequality of stochastic convolution can not be established as in Wiener noises case. Section \ref{s:ConSta} establishes this inequality using an integration by parts technique other than Da Prato-Kwapie\'n-Zabczyk's factorization technique,
which have their own interests\footnote{This section includes part of the author's not published results in \cite{Xu12}.}. This maximal inequality of stochastic $\alpha$-stable convolution is crucial for studying the well-posedness of
the mild solution, strong Feller property and accessibility property of the processes.
We hope that the results in this section will also be useful for studying the other SPDEs forced by
$\alpha$-stable type noises (\cite{DXZ12}). \\

{\bf Acknowledgements:} The author would like to gratefully thank Enrico Priola and Jerzy Zabczyk for leading him into
 the research field of SPDEs forced by L\'evy noises. He would also like to gratefully thank Zhao Dong and Xicheng Zhang for very useful
discussions.
\ \ \\


\section{Stochastic real Ginzburg-Landau equations driven by additive $\alpha$-stable noises}
Let $\T= \R/\Z$ be equipped with the usual Riemannian metric, and let $d \xi$
denote the Lebesgue measure on $\T$. Then
$$H:=\bigg\{x\in L^2(\T, \R): \int_\T x(\xi) d\xi =0\bigg\}$$
is a separable real Hilbert space with inner product
$$\Ll x,y \Rr_H:=\int_\T x(\xi)y(\xi) d\xi,\ \ \ \ \ \forall \ x, y \in H.$$
For $x\in C^2(\T)$, the Laplacian operator $\Delta$ is given by $\Delta x= x''.$
Let $(A, \mcl D(A))$ be the closure of $(-\Delta, C^2(\T)\cap H)$ in $H$,
which is a positively definite self-adjoint operator on $H$. \\

Denote $\Z_*:=\Z \setminus \{0\}$, $\{e_k\}_{k \in \Z_*}$ with $e_k=e^{i2 \pi k\xi}$ $(k \in \Z_*)$
is an orthonormal basis of $H$. For each $x \in H$,
it can represented by
$$x=\sum_{k \in \Z_*} x_k e_k;  \ \ \ \ x_k \in \mathbb C, x_{-k}=\overline{x_k}.$$
Write
$$\gamma_k:=4 \pi^2 |k|^2, \ \ \ \ \ \ k \in \Z_*,$$
it is easy to check $A e_k=\gamma_k e_k$ for all $k \in \Z_*$ and that
 $$\|A^\sigma x\|^2_H=\sum_{k \in \Z_*} |\gamma_k|^{2 \sigma} |x_k|^2, \ \ \ \ \ \ \ \sigma \ge 0,$$
provided the sum on the right hand side is finite.
Denote
$$V:= \mcl D(A^{1/2}),$$
it gives rise to a Hilbert space,
which is densely and compactly embedded in $H$. For each $x \in V$ with $x=\sum_{k \in \Z_*} x_k e_k$, we have
$$\|x\|^2_V=\sum_{k \in \Z_*} \gamma_k |x_k|^2.$$
Let $z(t)$ be a one-dimensional symmetric $\alpha$-stable process with $0<\alpha<2$. Its infinitesimal generator~$\mcl A$ is given by
\begin{equation} \label{e:fraclap}
\mcl A f(x):=\frac{1}{C_{\alpha}} \int_{\mathbb{R}}
\frac{f(y+x)-f(x)}{|y|^{\alpha+1}}dy, \;\; \ f \in C^2_b(\R),
\end{equation}
where $C_{\alpha}= -\int_{\mathbb{R}} (\cos y-1)\frac{dy}{|y|^{1+\alpha}}$;
see \cite{sato}.
It is well known that~$z(t)$ has the following characteristic function:
\begin{\eqn*}
\E [e^{i \lambda z(t)}]=e^{-t|\lambda|^{\alpha}},
\end{\eqn*}
$t \ge 0$, $\lambda \in \R$. \\

We shall study the 1D stochastic real Ginzburg-Landau equation on $\T$ as the following
\begin{equation} \label{e:XEqn}
\begin{cases}
dX_t + [A X_t + N(X_t)] dt = dL_t, \\
X_0=x,
\end{cases}
\end{equation}
where
\begin{itemize}
\item[(i)] The nonlinear term $N$ is defined by
\begin{equation*} \label{e:NonlinearB}
N(u)=-(u-u^3), \ \ \ \ \ u \in H.
\end{equation*}
\item[(ii)] $L_t=\sum_{k \in \Z_*} \beta_k l_k(t) e_k$ is a cylindrical $\alpha$-stable processes on $H$ with
$\{l_k(t)\}_{k\in \Z_*}$ being i.i.d. 1 dimensional symmetric $\alpha$-stable process sequence with $\alpha>1$. Moreover,
there exist some $C_1, C_2>0$ so that
$C_1 \gamma_k^{-\beta} \le |\beta_k| \le C_2 \gamma_k^{-\beta}$ with $\beta>\frac 12+\frac 1{2\alpha}$.
\end{itemize}
\begin{rem}
The condition $\beta>\frac 12+\frac 1{2\alpha}$ in (ii) guarantees that
the convolution $Z_t$ defined by \eqref{e:OUAlp} are in $V$.
\end{rem}

Let $C>0$ be some constant and let $C_p>0$ be some constant depending on some parameter $p$.
We shall often use the following inequalities:
\Be \label{e:PoiInq}
\|A^{\sigma_1} x\|_H \le  C_{\sigma_1, \sigma_2} \|A^{\sigma_2} x\|_H, \ \ \ \ \ \ \forall \ \sigma_1 \le \sigma_2, \ \forall x \in H;
\Ee
\Be \label{e:eAEst}
\|A^{\sigma} e^{-At}\| \le C_\sigma t^{-\sigma}, \ \ \ \ \ \forall \ \sigma>0;
\Ee
\Be \label{e:NInnPro}
\langle x, -N(x)\rangle_H \le \frac 14, \ \ \ \ \forall \ x \in H;
\Ee
\Be \label{e:NVEst}
\|N(x)\|_V \leq C (\|x\|_V+\|x\|^3_V), \ \ \ \ \forall \ x \in V;
\Ee
\Be \label{e:NxyHEst1}
\|N(x)-N(y)\|_H \le C(1+\|A^{\frac 14}x\|^2_H+\|A^{\frac 14}y\|^2_H)\|x-y\|_H, \ \ \ \forall \ x,y \in H.
\Ee
For all $\sigma \ge \frac 16$,
\Be \label{e:NxyHEst}
\|N(x)-N(y)\|_H \le C(1+\|A^{\sigma}x\|^2_H+\|A^{\sigma}y\|^2_H)\|A^{\sigma}(x-y)\|_H, \ \ \ \forall \ x,y \in H;
\Ee
\Be \label{e:NHEst}
\|N(x)\|_H \le C(1+\|A^{\sigma} x\|^3_H), \ \ \ \forall \ x \in H.
\Ee
We shall show \eqref{e:NInnPro}-\eqref{e:NHEst} in the appendix.\\

Let $E$ be a Banach space and let $T>0$ be arbitrary. Denote by $B_b(E)$ the space of
bounded measurable functions: $f: E \rightarrow \R$. Denote by $D([0,T];E)$ the space of the functions
$f: [0,T] \rightarrow E$ which has left limit and is right continuous. Denote by $C([0,T];E)$ the space of the functions
$f: [0,T] \rightarrow E$ which is continuous.
If $f \in D([0,T];E)$, it is said to be C$\grave{a}$dl$\grave{a}$g in $E$.
\\

The main results of this paper are the following three theorems.
\begin{thm} \label{t:MaiThm}
The following statements hold:
\begin{enumerate}
\item \label{1} For every $x \in H$ and $\omega \in \Omega$ a.s.,
Eq. \eqref{e:XEqn} admits a unique mild solution $X.(\omega) \in D([0,\infty);H) \cap D((0,\infty);V)$.
Moreover, $X_.(\omega)$ has the following form:
\Bes
X_t(\omega)=e^{-At} x+\int_0^t e^{-A(t-s)} N(X_s(\omega))ds+\int_0^t e^{-A(t-s)} d L_s(\omega), \ \ \ \forall \ t>0.
\Ees
\item \label{2} $X$ is a Markov process.
\item \label{3} For every $x \in V$ and $\omega \in \Omega$ a.s., we have $X_{.}(\omega) \in D([0,\infty);V)$.
Moreover, for every $T>0$,
\Bes
\sup_{0 \le t \le T} \|X_t(\omega)\|_V \le C,
\Ees
where $C$ is some constant depending on $T, \alpha, \beta$ and $\omega$.
\end{enumerate}
\end{thm}

\begin{thm} \label{t:Inv}
$X$ admits at least one invariant measure. The invariant measures are supported on $V$.
\end{thm}

\begin{thm} \label{t:Erg}
$X$ admits a unique invariant measure if $\alpha \in (3/2,2)$ and $\frac 12+\frac{1}{2\alpha}<\beta<\frac 32-\frac{1}{\alpha}$.
\end{thm}
\ \ \ \

\section{Stochastic convolution of $\alpha$-stable noises $(Z_t)_{t \ge 0}$} \label{s:ConSta}
Consider the following Ornstein-Uhlenbeck process:
\Be \label{e:OUAlp}
Z_t=\int_0^t e^{-A(t-s)} d L_s=\sum_{k \in \Z_{*}} z_{k}(t) e_k
\Ee
where $$z_{k}(t)=\int_0^t e^{-\gamma_k(t-s)}
\beta_k d l_k(s), \ \ \ \ \gamma_k=4 \pi^2 |k|^2.$$

We shall prove the following two lemmas in this section: the first one
is a maximal inequality of $(Z_t)_{t \ge 0}$,
while the other claims that $(Z_t)_{0 \le t \le T}$ stays, with positive probability,
in arbitrary small ball with zero center for all $T>0$. These two lemmas will play a crucial role in proving
strong Feller and accessibility for the solution of Eq. \eqref{e:XEqn}. \cite{BrHa09} established a nice
maximal inequality for the stochastic convolution of a family of L\'evy noises, but these noises
do not include $\alpha$-stable noises.

\begin{lem} \label{l:ZEst}
Let $T>0$ be arbitrary.
For all $0 \leq \theta<\beta-\frac 1{2 \alpha}$ and all $0<p<\alpha$, we have
\Be
\E \sup_{0 \leq t \le T}\|A^\theta Z_t\|^p_{H} \le CT^{p/\alpha},
\Ee
where $C$ depends on $\alpha,\theta, \beta, p$.
\end{lem}

\begin{lem} \label{l:Acc}
Let $\tl \theta \in [0,\beta-\frac 1{2 \alpha})$ be arbitrary. For all $T>0$ and $\e>0$, we have
\Bes
\PP(\sup_{0 \le t\le T} \|A^{\tl \theta} Z_t\|_H \le \e)>0.
\Ees
\end{lem}
\ \ \

\begin{proof} [Proof of Lemma \ref{l:ZEst}]
We only need to show the inequality for the case $p \in (1,\alpha)$ since the
case of $0<p \le 1$ is an immediate corollary by H$\ddot{o}$lder's inequality.
\\

\emph{Step 1.} We claim that for all $\theta \in [0,\beta-\frac 1{2\alpha})$, $\|A^\theta L_t\|_H$ is a right continuous submartingale such that
\Be \label{e:EALT}
\E\|A^{\theta} L_t\|_H^p \le C_{\alpha, p} \left(\sum_{k \in \Z_*} |\beta_k|^\alpha \gamma_k^{\alpha \theta}\rt)^{p/\alpha} t^{p/\alpha}, \ \ \ \ \ \ t>0.
\Ee
where $p \in (1,\alpha)$.
 \\

 We first follow the argument in the proof of \cite[Theorem 4.4]{PZ11} to show \eqref{e:EALT}. Take a Rademacher
sequence $\{r_k\}_{k \in \Z_*}$ in a new probability space
$(\Omega^{'},\mcl F^{'},\PP^{'})$, i.e. $\{r_k\}_{k \in \Z_*}$ are
i.i.d. with $\PP\{r_k=1\}=\PP\{r_k=-1\}=\frac 12$. By the following
Khintchine inequality: for any $p>0$, there exists some $C(p)>0$
such that for arbitrary real sequence $\{h_k\}_{k \in \Z_*}$,
$$\left(\sum_{k \in \Z_*} h^2_k\right)^{1/2} \leq C(p) \left(\E^{'} \left|\sum_{k \in \Z_*} r_k h_k\right|^p\right)^{1/p}.$$
By this inequality, we get
\begin{\eqn} \label{e:EZAtp}
\begin{split}
\E\|A^{\theta} L_t\|^p_H&=\E \left(\sum_{k \in \Z_*} \gamma_k^{2\theta}
|\beta_k|^2|l_k(t)|^2\right)^{p/2} \leq C \E \E^{'}\left|\sum_{k \in \Z_*} r_k \gamma_k^{\theta}
|\beta_k| l_k(t)\right|^p \\
&=C\E^{'}\E\left|\sum_{k \in \Z_*} r_k \gamma_k^{\theta}
|\beta_k| l_k(t)\right|^p
\end{split}
\end{\eqn}
where $C=C^p(p)$. For any $\lambda \in \R$, by the fact $|r_k|=1$
and an approximation argument similar as for getting (4.12) of \cite{PZ11}, one has
\begin{\eqn*}
\begin{split}
\E\exp \left\{i \lambda \sum_{k \in \Z_*} r_k \gamma_k^{\theta}
|\beta_k| l_k(t)\right\}&=\exp\left\{-|\lambda|^\alpha \sum_{k \in \Z_*}
|\beta_k|^\alpha \gamma_k^{\alpha \theta}t
 \right\}
\end{split}
\end{\eqn*}
\vskip 3mm
Now we use (3.2) in \cite{PZ11}: if $X$ is a symmetric random
variable satisfying $$\E \left[e^{i\lambda
X}\right]=e^{-\sigma^{\alpha} |\lambda|^\alpha}$$ for some $\alpha
\in (0,2)$ and any $\lambda \in \R$, then for all $p \in (0,\alpha)$,
$$\E|X|^p=C(\alpha,p)
\sigma^p. $$
Since $\sum_{k \in \Z_*}
|\beta_k|^\alpha \gamma_k^{\alpha \theta}<\infty$, \eqref{e:EALT} holds. \\

Let us now show that $\|A^\theta Z_t\|_H$ is a right continuous submartingale. Denote $L^n_t=\sum_{|k| \le n} \beta_k l_k(t) e_k$ for all
$n \in \N$ and $t>0$, it is clear that $A^\theta L^n_t$ is an $L^p$ martingale with $p \in (1,\alpha)$. Therefore, $\|A^\theta L^n_t\|_H$ is a submartingale, i.e.,
\Bes
\E[\|A^\theta L^n_t\|_H|\mcl F_s] \ge \|A^\theta L^n_s\|_H, \ \ \ t>s.
\Ees
Thanks to the inequality \eqref{e:EALT} with some $p>1$, let $n \rightarrow\infty$, we get
\Bes
\E[\|A^\theta L_t\|_H|\mcl F_s] \ge \|A^\theta L_s\|_H,
\Ees
i.e., $\|A^\theta Z_t\|_H$ is a submartingale.
Observe that $A^\theta L_t=\sum_{k \in \Z^*} \gamma^\theta_k \beta_k l_k(t) e_k$. Since
\Bes
\sum_{k \in \Z_*}\gamma^{\alpha\theta}_k |\beta_k|^\alpha <\infty,
\Ees
it follows from \cite[Theorem 2.2]{LiuZha12} that $A^{\theta} L_t$ has a C\`{a}dl\`{a}g version. Hence, $\|A^\theta L_t\|_H$
has a right continuous version. \\

\emph{Step 2}.
It follows from Ito's product formula (\cite[Theorem 4.4.13]{Ap04}) that for all $k \in \Z_*$,
\Bes
l_k(t)=\int_0^t \gamma_k e^{-\gamma_k(t-s)} l_k(s)ds+\int_0^t e^{-\gamma_k(t-s)} dl_k(s)+\int_0^t \gamma_k e^{-\gamma_k(t-s)} \Delta l_k(s) ds
\Ees
where $\Delta l_k(s)=l_k(s)-l_k(s-)$. Since $l_k(t)$ is an $\alpha$-stable process, $\Delta l_k(s)=0$ for $s \in [0,t]$ a.s. and thus
\Bes
\int_0^t \gamma_k e^{-\gamma_k(t-s)} \Delta l_k(s) ds=0.
\Ees
Therefore,
\Be \label{e:zkIBP}
\begin{split}
z_k(t)&=\beta_k l_k(t)-\int_0^t \gamma_k e^{-\gamma_k(t-s)} \beta_k l_k(s)ds
\end{split}
\Ee
Hence, we get
\Be \label{e:IntByPar}
\begin{split}
Z_t=L_t-Y_t,
\end{split}
\Ee
where
\Be \label{e:YtInt}
Y_t=\int_0^t A e^{-A(t-s)} L_s ds.
\Ee

By Step 1 and Doob's martingale inequality, we have
\Be \label{e:DooIne}
\E\lt[\sup_{0 \leq t \le T} \lf \|A^{\theta} L_t\rg \|^p_H \rt] \le \lt(\frac p{p-1}\rt)^p \E\lt[\lf \|A^{\theta} L_T\rg \|^p_H \rt]
\le C T^{\frac p\alpha}, \ \ \ \ \theta \in [0,\beta-\frac 1{2 \alpha}),
\Ee
where $C$ depends on $\theta, \alpha, \beta$ and $p$.
Choose some $0<\e<\min\{\beta-\frac{1}{2 \alpha}-\theta,1\}$, we have
\Be \label{e:YtIntEst}
\begin{split}
\sup_{0 \le t \le T}\|A^{\theta} Y_t\|_H  & \le \sup_{0 \le t \le T} \int_0^t \|A^{1-\e} e^{-A(t-s)}\|\|A^{\theta+\e} L_s\|_H ds  \\
& \le \sup_{0 \le t \le T}  \|A^{\theta+\e} L_t\|_H \sup_{0 \le t \le T} \int_0^t \|A^{1-\e} e^{-A(t-s)}\| ds. \\
\end{split}
\Ee
If $T \le 1$, using \eqref{e:eAEst} we get
\Bes
\sup_{0 \le t \le T} \int_0^t \|A^{1-\e} e^{-A(t-s)}\| ds \le C\sup_{0 \le t \le T} \int_0^t (t-s)^{-1+\e} ds \le C/\e.
\Ees
If $T>1$, using \eqref{e:eAEst} again we get
\Bes
\begin{split}
\sup_{0 \le t \le T} \int_0^t \|A^{1-\e} e^{-A(t-s)}\| ds &\le \sup_{0 \le t \le 1} \int_0^t \|A^{1-\e} e^{-A(t-s)}\| ds+\sup_{1 \le t \le T}\int_0^t \|A^{1-\e} e^{-A(t-s)}\| ds \\
& \le C/\e+\sup_{1 \le t \le T} \int_{t-1}^t \|A^{1-\e}e^{-A}\| \|e^{-A(t-1-s)}\| ds \\
& \ \ \ +\sup_{1 \le t \le T} \int_0^{t-1} \|A^{1-\e}e^{-A}\| \|e^{-A(t-1-s)}\| ds \\
& \le 2C/\e+\|A^{1-\e}e^{-A}\|\sup_{1 \le t \le T} \int_0^{t-1} e^{-2 \pi (t-1-s)} ds,
\end{split}
\Ees
where the last inequality is by the spectral property of $A$.
Collecting above three inequalities, we have
\Be
\sup_{0 \le t \le T}\|A^{\theta} Y_t\|_H \le C_\e \sup_{0 \le t \le T}  \|A^{\theta+\e} L_t\|_H.
\Ee
This, together with \eqref{e:DooIne}, implies
\Be \label{e:IEst}
\E \sup_{0 \le t \le T}\|A^{\theta} Y_t\|^p_H \le C T^{p/\alpha},
\Ee
where $C$ depends on $\alpha, \beta, \theta, \e$ and $p$. \\

Combining the above inequality with \eqref{e:DooIne} and \eqref{e:IntByPar}, we immediately get the desired inequality.
\end{proof}
\ \ \ \ \ \

\begin{proof} [Proof of Lemma \ref{l:Acc}]
Take $0<\tl \theta<\theta<\beta-\frac{1}{2 \alpha}$,
since $\{z_k(t)\}_{k \in \Z_*}$ are independent sequence, we have
\Bes
\begin{split}
\PP\left(\sup_{0 \le t \le T} \|A^{\tl \theta}Z_t\|_H \le \e \right) &=
\PP \left(\sup_{0 \le t\le T} \sum_{k \in \Z_*}\gamma_k^{2\tl \theta} |z_k(t)|^2 \le \e^2\right) \ge I_1I_2
\end{split}
\Ees
where
$$I_1:=\PP \lt(\sup_{0 \le t\le T} \sum_{|k|>N} \gamma_k^{2\tl \theta} |z_k(t)|^2 \le \e^2/2\rt),$$
$$I_2:=\PP \lt(\sup_{0 \le t\le T} \sum_{|k| \le N} \gamma_k^{2\tl \theta} |z_k(t)|^2 \le \e^2/2\rt),$$
with $N\in \N$ being some fixed large number.
By the spectral property of $A$, we have
\Bes
\begin{split}
I_1 & \ge \PP \lt(\sup_{0 \le t\le T} \sum_{|k|>N}\gamma_k^{2\tl \theta} |z_k(t)|^2 \le \e^2/2\rt) \\
& \ge \PP \lt(\sup_{0 \le t\le T} \|A^\theta Z(t)\|^2_H \le \gamma_N^{2(\theta-\tl \theta)} \e^2/2\rt) \\
&=1-\PP \lt(\sup_{0 \le t\le T} \|A^\theta Z(t)\|^2_H>\gamma_N^{2(\theta-\tl \theta)} \e^2/2\rt) \\
&=1-\PP \lt(\sup_{0 \le t\le T} \|A^\theta Z(t)\|^p_H>\gamma_N^{p(\theta-\tl \theta)} \e^p/2^{p/2}\rt)
\end{split}
\Ees
This, together with Lemma \ref{l:ZEst} and Chebyshev inequality, implies that as $\gamma_N$ is sufficient large
\Bes
I_1 \ge 1-C\gamma_N^{-(\theta-\tl \theta)p} \e^{-p}>0,
\Ees
where $p \in (1, \alpha)$ and $C$ depends on $p, \alpha, \beta, T$. \\

To finish the proof, it suffices to show that
\Be \label{e:IrrI2>0}
I_2>0.
\Ee
Define $A_k:=\{\sup_{0 \le t\le T} |z_k(t)| \le \e/(\sqrt {2N} \gamma^{\tl \theta}_k)\}$,
it is easy to have
\Be \label{e:IrrI2Cal}
I_2  \ge \PP\lt(\bigcap_{|k| \le N} A_k\rt)=\prod_{|k| \le N} \PP(A_k).
\Ee
Recalling \eqref{e:zkIBP}, we have
\Bes
z_k(t)=\beta_k l_k(t)-\int_0^t \gamma_k e^{-\gamma_k (t-s)} \beta_k l_k(s) ds.
\Ees
Furthermore, it follows from a straightforward calculation that
\Bes
\sup_{0 \le t \le T} |\int_0^t \gamma_k e^{-\gamma_k(t-s)} \beta_k l_k(s) ds| \le |\beta_k| \sup_{0 \le t \le T} |l_k(t)|
\ \ \ \ \ \ \ k \ge 1.
\Ees
Therefore,
\Bes
\PP(A_k) \ge \PP \lt(\sup_{0 \le t\le T} |l_k(t)| \le \frac{\e}{2 |\beta_k| \sqrt{2N} \gamma^{\tl \theta}_k}\rt)
\Ees
By \cite[Proposition 3, Chapter VIII]{Ber96}, there exist some
$c, C>0$ only depending on $\alpha$ so that
\Bes
\PP(\sup_{0 \le t \le T} |l(t)| \le 1) \ge C e^{-cT}.
\Ees
This, together with the scaling property of stable process, implies
\Be
\PP(A_k)>0 \ \ \ \ |k| \le N,
\Ee
which, combining with \eqref{e:IrrI2Cal}, immediately implies \eqref{e:IrrI2>0}.
\end{proof}
\ \ \ \ \ \

By Lemma \ref{l:ZEst} above and Theorem 2.2 in \cite{LiuZha12}, we have the following lemma which is important for showing the solution of
Eq. \eqref{e:XEqn} has C\`{a}dl\`{a}g trajectories.
\begin{lem} \label{l2.2}
$Z.$ defined by \eqref{e:OUAlp} has a version in $D([0, \infty);V)$.
\end{lem}
\begin{proof}
Observe that $A^{\frac 12} L_t=\sum_{k \in \Z^*} \gamma^{\frac 12}_k \beta_k l_k(t) e_k$. By Lemma \ref{l:ZEst} with
$\theta=\frac 12$, we have
\Bes
\sum_{k \in \Z_*}\gamma^{\frac 12\alpha}_k |\beta_k|^\alpha <\infty.
\Ees
This, together with \cite[Theorem 2.2]{LiuZha12}, implies that $Z_t$ has a version which has left limit and is right continuous in $V$.
\end{proof}

\section{Proof of Theorem \ref{t:MaiThm}} \label{s:Y}

For all $\omega \in \Omega$, define $Y_t(\omega):=X_t(\omega)-Z_t(\omega)$, then
\Be \label{e:YEqn}
\p_t Y_t(\omega)+AY_t(\omega)+N(Y_t+Z_t)(\omega)=0,
\ \ \ Y_0(\omega)=x.
\Ee
For each $T>0$, define
\Be \label{e:KT}
K_T(\omega):=\sup_{0 \le t \le T}\|Z_t(\omega)\|_V, \ \ \ \ \ \ \omega \in \Omega.
\Ee
From Lemma \ref{l:Acc}, for every $k \in \N$, there exists some set $N_k \subset \Omega$ such that $\PP(N_k)=0$ and
\Bes
K_k(\omega)<\infty, \ \ \ \ \omega \notin N_k.
\Ees
Define $N=\cup_{k \ge 1} N_k$, it is easy to see $\PP(N)=0$ and that for all $T>0$
\Be \label{e:KTFin}
\begin{split}
K_T(\omega)<\infty,   \ \ \ \ \ \ \omega \notin N.
\end{split}
\Ee

\begin{lem} \label{l:LocExUnH}
The following statements hold:
\begin{itemize}
\item[(i)] For every $x \in H$ and $\omega \notin N$,
 there exists some $0<T(\omega) \le 1$, depending on $\|x\|_H$ and $K_1(\omega)$, such that Eq. \eqref{e:YEqn} admits a unique solution
$Y.(\omega) \in C([0,T];H)$ satisfying for all $\sigma \in [\frac 16,\frac 12]$
\Bes
\|A^{\sigma} Y_t(\omega)\|_H \le C(t^{-\sigma}+1), \ \ \ \ \ t \in (0,T(\omega)],
\Ees
where $C$ is some constant depending on $\|x\|_H, \sigma$ and $K_1(\omega)$.
\item[(ii)] Let $\sigma \in [\frac 16, \frac 12]$. For every $x \in \mcl D(A^{\sigma})$ and $\omega \notin N$,
there exists some $0<\hat T(\omega) \le 1$, depending on $\|x\|_H, \sigma, K_1(\omega)$, such that Eq. \eqref{e:YEqn} admits
a unique solution on $C([0,\hat T(\omega)];\mcl D(A^{\sigma}))$ such that
\Bes
\sup_{0 \le t \le \hat T(\omega)} \|A^{\sigma} Y_t(\omega)\|_H \le 1+\|A^{\sigma} x\|_H.
\Ees
In particular, when $\sigma=1/2$,
\Bes
\sup_{0 \le t \le \hat T(\omega)} \|Y_t(\omega)\|_V \le 1+\|x\|_V.
\Ees
\end{itemize}
\end{lem}

\begin{proof}
We shall omit the variable $\omega$ for the notational simplicity in the proof,, since no confusions will arise.\\

(i). We shall apply Banach fixed point theorem. Let $0<T \le 1$ and $B>0$
be some constants to be determined later. Take $\sigma=\frac 16$ and define
$$S=\{u \in C([0,T];H): u_0=x, \sup_{0 \le t \le T} t^{\sigma} \|A^{\sigma} u_t\|_H \le B\}.$$
Given any $u, v \in S$, define $d(u,v)=\sup_{0 \le t \le T} t^{\sigma} \|A^{\sigma} (u_t-v_t)\|_H,$
then $(S, d)$ is a closed metric space. Define a map $\mcl F: S \rightarrow C([0,T];H)$ as the following: for any $u \in S$,
$$(\mcl F u)_t=e^{-At} x+\int_0^t e^{-A(t-s)} N(u_s+Z_s)ds, \ \ \ \ t \in [0,T],$$
we aim to show as $T$ is sufficient small and $B$ is sufficiently large,
\begin{itemize}
\item[(a)]  $\mcl Fu \in S$ for $u \in S$,
\item[(b)]  $d(\mcl F u,\mcl F v) \le \frac 12 d(u,v)$ for $u, v\in S$.
\end{itemize}
\vskip 3mm

It is obvious $(\mcl F u)_0=x$. By \eqref{e:eAEst}, \eqref{e:NHEst}, \eqref{e:PoiInq} and Young's inequality, we have
\Bes 
\begin{split}
\|A^{\sigma} (\mcl F u)_t\|_H
& \le  Ct^{-\sigma}\|x\|_H+\int_0^t \|A^{\sigma} e^{-A(t-s)}\|\|N(u_s+Z_s)\|_H ds \\
& \le Ct^{-\sigma}\|x\|_H+C\int_0^t (t-s)^{-\sigma}(1+\|A^{\sigma}u_s+A^{\sigma} Z_s\|^3_H) ds \\
&\le Ct^{-\sigma}\|x\|_H+C\int_0^t (t-s)^{-\sigma}(1+K_1^3+\|A^{\sigma}u_s\|^3_H) ds.
\end{split}
\Ees
Hence,
\Bes
t^{\sigma} \|A^{\sigma} (\mcl F u)_t\|_H \le C \|x\|_H+Ct^{\sigma}\int_0^t (t-s)^{-\sigma} (1+K_1^3+s^{-3 \sigma} B^3) ds
\Ees
As $T>0$ is sufficiently small and $B$ is sufficiently large, (a) immediately follows from the above inequality.
\vskip 3mm

Given any $u, v \in S$, it follows from \eqref{e:eAEst} and \eqref{e:NxyHEst}
\Bes
\begin{split}
& \quad \ t^{\sigma}\|A^\sigma (\mcl F u)_t-A^\sigma (\mcl F v)_t\|_H \\
& \le C T^{\sigma} \int_0^t (t-s)^{-\sigma} \|N(u_s+Z_s)-N(v_s+Z_s)\|_H ds \\
& \le C T^{\sigma} \int_0^t (t-s)^{-\sigma}(1+K_1^2+\|A^{\sigma}u_s\|^2_H+\|A^{\sigma}v_s\|^2_H) \|A^\sigma u_s-A^\sigma v_s\|_H ds \\
& \le C(1+K_1^2) T^{\sigma} \int_0^t (t-s)^{-\sigma} s^{-\sigma} \left[s^{\sigma} \|A^\sigma u_s-A^\sigma v_s\|_H\right] ds \\
& \quad+C T^{\sigma} \int_0^t (t-s)^{-\sigma} s^{-3 \sigma} B^2 \left[s^{\sigma} \|A^\sigma u_s-A^\sigma v_s\|_H\right] ds \\
\end{split}
\Ees
where the last inequality is by the fact $u, v \in S$. This inequality implies
\Bes
\begin{split}
& \quad \  \sup_{0 \le t \le T} t^{\sigma}\|A^\sigma (\mcl F u)_t-A^\sigma (\mcl F v)_t\|_H \\
& \le C[(1+K_1^2)T^{1-\sigma}+B^2 T^{\sigma}] \sup_{0 \le t \le T} t^{\sigma}\|A^\sigma u_t-A^\sigma v_t\|_H.
\end{split}
\Ees
Choosing $T$ small enough, we immediately get (b) from the above inequality. Combining (a) and (b), Eq. \eqref{e:YEqn} has a unique solution in $S$ by Banach fixed point theorem.
\vskip 3mm

Let $Y. \in S$ be the solution obtained by the above Banach fixed point theorem, for every $\sigma \in [\frac 16, \frac 12]$ and $t \in (0,T]$, by \eqref{e:eAEst} and \eqref{e:NHEst} we have
\Bes
\begin{split}
\|A^\sigma Y_t\|_H & \le Ct^{-\sigma} \|x\|_H+C\int_0^t (t-s)^{-\sigma}\|N(Y_s+Z_s)\|_H ds \\
& \le Ct^{-\sigma} \|x\|_H+C\int_0^t (t-s)^{-\sigma} (\|A^{\frac 16} Y_s\|^3+K_1^3)ds \\
& \le Ct^{-\sigma} \|x\|_H+C\int_0^t (t-s)^{-\sigma} s^{-\frac 12}(B^3+K_1^3)ds.
\end{split}
\Ees
This inequality clearly imply the desired inequality.\\

Let us show the uniqueness. Let $u, v \in C([0,T];H)$ be two solutions satisfying the inequality, it follows from \eqref{e:NxyHEst1}
that for all $t \in [0,T]$,
\Bes
\begin{split}
\|u_t-v_t\|_H & \le \int_0^t \|N(u_s+Z_s)-N(v_s+Z_s)\|_H ds \\
& \le \int_0^t (1+K_1^2+\|A^{\frac 14}u_s\|^2_H+\|A^{\frac 14}v_s\|^2_H) \|u_s-v_s\|_H ds \\
& \le \int_0^t (1+K_1^2+C s^{-1/2}) \|u_s-v_s\|_H ds,
\end{split}
\Ees
 the above inequality implies $\|u_t-v_t\|_H=0$ for all $t \in [0,T]$.
\\

(ii). Let $0<\hat T \le 1$ be some constant to be determined later. For every $\sigma \in [\frac 16, \frac 12]$, define
$$S=\{u \in C([0,\hat T];\mcl D(A^{\sigma})): u_0=x, \sup_{0 \le t \le \hat T} \|A^{\sigma} u_t\|_H \le 1+\|A^{\sigma} x\|_H\}.$$
Given any $u, v \in S$, define
$d(u,v)=\sup_{0 \le t \le \hat T} \|A^{\sigma} (u_t-v_t)\|_H,$ then $(S, d)$ is a closed metric space.
\vskip 2mm
Define a map $\mcl F: S \rightarrow C([0,\hat T]; D(A^\sigma))$ as the following: for any $u \in S$,
$$(\mcl F u)_t=e^{-At} x+\int_0^t e^{-A(t-s)} N(u_s+Z_s)ds, \ \ \ \ 0 \le t \le \hat T.$$
By \eqref{e:eAEst} and \eqref{e:NHEst}, we have
\Bes
\begin{split}
\|A^{\sigma}(\mcl F u)_t\|_H
& \le \|A^{\sigma} x\|_H+C\int_0^t (t-s)^{-\sigma} (1+\|A^{\sigma} u_s\|_H^3+K^3_1) ds \\
& \le \|A^{\sigma} x\|_H+C\int_0^t (t-s)^{-\sigma} (1+\|A^{\sigma} x\|_H^3+K^3_1) ds
\end{split}
\Ees
Choosing $\hat T>0$ so small that $\|A^{\sigma}(\mcl F u)_t\|_H \le 1+\|A^{\sigma} x\|_H$ for all $0 \le t \le \hat T$, we get
$\mcl F: S \rightarrow S$ from the previous inequality.
\vskip 2mm
For all $u, v \in S$, as $\hat T>0$ is sufficiently small, by a similar calculation we have
$$\sup_{0 \le t \le \hat T} \|A^{\sigma}(\mcl F u)_t-A^{\sigma}(\mcl F v)_t\|_H \le \frac 12 \sup_{0 \le t \le \hat T} \|A^{\sigma}u_t-A^{\sigma}v_t\|_H,$$
i.e., $d(\mcl F u,\mcl F v) \le \frac 12 d(u,v)$.
By Banach fixed point theorem, we complete the proof of (ii).
\end{proof}
\ \ \ \
\begin{lem} \label{l:YGlExUn}
The following statements hold:
\begin{itemize}
\item[(a)] For every $x \in H$ and $\omega \notin N$, Eq. \eqref{e:YEqn} admits a unique global solution on
$Y.(\omega) \in C([0,\infty);H) \cap C((0,\infty);V)$.
\item[(b)] If $x \in V$, then $Y_{.}(\omega) \in C([0,\infty);V)$.
\end{itemize}
\end{lem}

\begin{proof}
For national simplicity, we shall omit the variable $\omega$ in the proof. Thanks to \eqref{e:KT}, \eqref{e:KTFin}
and Lemma \ref{l:LocExUnH},
to get a global solution, it suffices to show the following a'priori estimate:
\Be \label{e:PriEst}
\|Y_t\|_H^2 \le e^{-(2 \pi-3)t} \|x\|^2_H+\int_0^t e^{-(2\pi-3)(t-s)} (\|Z_s\|^2_H+C\|Z_s\|_V^4) ds.
\Ee

To this end, let us first show the following auxiliary inequality:
\Be \label{e:NuvuH}
\Ll -N(u+v), u\Rr_H \le \frac 32\|u\|_H^2+\frac 12\|v\|_H^2+C\|v\|_V^4.
\Ee
In fact, it follows from the following Young's inequalities: for $a,b \ge 0$,
$$ab \le \frac {a^2} 2+\frac {b^2} 2, \ \ \  ab \le \frac{a^4}{4}+\frac{3b^{\frac 43}}{4},$$
that
\Bes
\begin{split}
\Ll -N(u+v), u\Rr_H&=\int_\T |u(\xi)|^2 d\xi+\int_\T u(\xi)v(\xi)d\xi-\int_\T |u(\xi)|^4 d\xi\\
&\quad -3\int_\T u^3(\xi)v(\xi) d\xi
 -3 \int_\T u^2(\xi)v^2(\xi) d\xi-\int_\T u(\xi)v^3(\xi) d\xi \\
& \le \|u\|_H^2+\frac 12\|u\|_H^2+\frac 12\|v\|_H^2-\|u^2\|^2_H+\|u^2\|_H^2+C\|v^2\|_H^2 \\
&=\frac 32\|u\|_H^2+\frac 12\|v\|_H^2+C\|v^2\|_H^2.
\end{split}
\Ees
This, together with Sobolev embedding $\|v^2\|^2_H \le C\|v\|_V^4$, immediately implies \eqref{e:NuvuH}.
\vskip 2mm
It follows from \eqref{e:NuvuH} and Poincare inequality $\|x\|_H \le \frac {1}{2 \pi}\|x\|_V$ that
\Bes
\begin{split}
\|Y_t\|_H^2&=\|x\|_H^2-2\int_0^t \|Y_s\|_V^2 ds-2\int_0^t \Ll N(Y_s+Z_s), Y_s\Rr_H ds \\
& \le \|x\|_H^2-(2\pi-3) \int_0^t \|Y_s\|^2_H ds+\int_0^t \|Z_s\|^2_H ds+C \int_0^t \|Z_s\|_V^4 ds.
\end{split}
\Ees
This implies \eqref{e:PriEst} immediately.
\vskip 2mm

It follows from Lemma \ref{l:LocExUnH} that
Eq. \eqref{e:YEqn} admits a unique local solution $Y. \in C([0,T];H) \cap C((0,T];V)$ for some $T>0$.
Thanks to (ii) of Lemma \ref{l:LocExUnH} and \eqref{e:PriEst}, we can extend this solution $Y. \in C([0,T];H) \cap C((0,T];V)$
to be $Y. \in C([0,\infty);H) \cap C((0,\infty);V)$.
\vskip 2mm

We now prove the new $Y. \in C([0,\infty);H) \cap C((0,\infty);V)$ is unique. Suppose there are two solutions $Y.^1, Y.^2 \in C([0,\infty);H) \cap C((0,\infty);V)$. Thanks to the uniqueness of $Y_.$ on $[0,T]$, we have
$Y^1_T=Y^2_T$. For any $T_0>T$, it follows from the continuity that
\Bes
\sup_{T \le t \le T_0} \|Y^1_t\| \le \hat C, \ \ \ \sup_{T \le t \le T_0} \|Y^2_t\| \le \hat C,
\Ees
where $\hat C>0$ depends on $T_0,Y^1,Y^2$ and $\omega$. Hence, for all $t \in [T,T_0]$, by \eqref{e:eAEst}, \eqref{e:PoiInq} and \eqref{e:NxyHEst1},
\Bes
\begin{split}
\|Y^1_t-Y^2_t\|_V & \le \int_{T}^t (t-s)^{-1/2}\|N(Y^1_s+Z_s)-N(Y^2_s+Z_s)\|_H ds \\
& \le C\int_{T}^t (t-s)^{-1/2} (1+K_1^2+\|A^{\frac 14}Y^1_s\|^2_H+\|A^{\frac 14}Y^2_s\|^2_H) \|Y^1_s-Y^2_s\|_H ds \\
& \le C \hat K \int_{T}^t (t-s)^{-1/2}  \|Y^1_s-Y^2_s\|_V ds,
\end{split}
\Ees
where $\hat K:=(1+K^2_1+2 \hat C^2)$.
This immediately implies $Y^1_t=Y^2_t$ for all $t \in [T,T_0]$. Since $T_0$ is arbitrary,
we get the uniqueness of the solution $Y_. \in C([0,\infty);H) \cap C((0,\infty);V)$.
\vskip 2mm
If $x \in V$, it follows from (a) that Eq. \eqref{e:YEqn} admits a unique solution $Y. \in C([0,\infty);H) \cap C((0,\infty);V)$.
By (ii) of Lemma \ref{l:LocExUnH}, Eq. \eqref{e:YEqn} admits a unique solution $\hat Y. \in C([0,\hat T];V)$ for some $\hat T>0$.
Since $C([0,\hat T];V)$ is a subset of $C([0,\hat T];H) \cap C((0,\hat T];V)$, $Y_t=\hat Y_t$ for all $t \in [0, \hat T]$. Hence, $Y. \in C([0,\infty);V)$.
\end{proof}
\ \ \ \ \
\begin{proof} [Proof of Theorem \ref{t:MaiThm}]
Let us study Eqs. \eqref{e:XEqn} and \eqref{e:YEqn} for $\omega \notin N$, where $N \subset \Omega$ is a negligible set defined in \eqref{e:KTFin}.
Since $Z.(\omega) \in D([0,\infty);V)$, it is of course
$Z.(\omega) \in D([0,\infty);H)$. By (a) of Lemma \ref{l:YGlExUn}, $X.(\omega)=Y.(\omega)+Z.(\omega)$ is the unique solution to Eq. \eqref{e:XEqn} in $D([0,\infty);H)\cap D((0,\infty);V)$.
The Markov property follows from the uniqueness immediately. (3) immediately follows from \eqref{e:KTFin} and (b) of Lemma \ref{l:YGlExUn}.
\end{proof}
\ \ \

\section{Proof of Theorem \ref{t:Inv}}
To show the existence of invariant measures, we follow the method in \cite{DXZ11}. To this end, let us consider the Galerkin approximation of
Eq. \eqref{e:XEqn}.

Recall that $\{e_k\}_{k \in \Z_{*}}$
is an orthonormal basis of $H$, define
$$H_m:={\rm span}\{e_k; |k| \le m\}$$
equipped with the norm  adopted from $H$. It is clear that $H_m$ is a
finite dimensional Hilbert space. Given any $m>0$, let $\pi_m: H \To H_m$ be the projection
from $H$ to $H_m$.

It is well known that for all fixed $m \in \N$, $H_m$ and $V_m$ are equivalent since we have
$$C_1 \|x\|_H \le \|x\|_V  \le C_2 \|x\|_H \ \ \ \ \ \forall \ x \in H_m,$$
where $C_1, C_2$ are both only depends on $m$.
\vskip 3mm

The Galerkin approximation of \eqref{e:XEqn} is as the following:
\Be \label{e:GalEqn}
d X^m_t+[A X^m_t+N^m(X^m_t)]dt=dL^m_t, \ \ \ X^m_0=x^m,
\Ee
where $X^m_t=\pi_m X_t$, $N^m(X^m_t)=\pi_m[N(X^m_t)]$,
$L^m_t=\sum_{|k| \le m} \beta_k l_k(t) e_k$.
Eq. \eqref{e:GalEqn} is a dynamics evolving in $H_m$.

\begin{thm} \label{t:GalCon}
The following statements hold:
\begin{enumerate}
\item[(i)] For every $x \in W$ with $W=H,V$ and $\omega \in \Omega$ a.s., Eq. \eqref{e:GalEqn} has a unique mild solution $X^{m,x^m}_.(\omega) \in D([0,\infty),W_m)$ such that
$$\sup_{0 \le t \le T} \|X^{m, x^m}_t(\omega)\|_W \le C, \ \ \ \ \ T>0,$$
where $C$ depends on $\|x\|_W, T$ and $K_T(\omega)$.
\item[(ii)] For every $x \in W$ with $W=H,V$ and $\omega \in \Omega$ a.s., we have
\Bes
\lim_{m \rightarrow \infty} \|X^{m,x^m}_t(\omega)-X^x_t(\omega)\|_W=0, \ \ \ \ \ \ \ t \ge 0.
\Ees
\end{enumerate}
\end{thm}

\begin{proof}
We omit the variable $\omega$ for the notational simplicity in the proof.
By the same method as proving Theorem \ref{t:MaiThm}, we can show (i). It remains to show (ii). Since the two cases $W=H$ and $W=V$ can be shown by the same method, we only prove the case $W=V$.
\vskip 2mm

As $t=0$, (ii) is obvious. For $t>0$, by (3) of Theorem \ref{t:MaiThm} and (i) we have
\Be \label{e:UniBouXmX}
\sup_{0 \le s \le t} \|X_s\|_V \le \hat C, \ \ \ \ \ \sup_{0 \le s \le t} \|X^m_s\|_V \le \hat C,
\Ee
where $\hat C>0$ depends on $\|x\|_V, t$ and $K_t$.
Observe
\Be
X^m_t-X_t=I_1(t)+I_2(t)+I_3(t)+I_4(t),
\Ee
where $I_1(t):=e^{-At}(x^m-x)$, $I_2(t):=Z_t-Z^m_t$,
\Bes
\begin{split}
& I_3(t):=\int_0^t e^{-A(t-s)} (I-\pi_m) N(X_s)ds,\\
& I_4(t):=\int_0^t e^{-A(t-s)} \left[N^m(X^m_s)-N^m(X_s)\right]ds.
\end{split}
\Ees
It is clear that as $m \rightarrow \infty$,
$$\|I_1(t)\|_V \rightarrow 0, \ \ \ \|I_2(t)\|_V \rightarrow 0.$$

\noindent
By \eqref{e:NHEst}, we have
\Be
\|(I-\pi_m)N(X_s)\|_H \rightarrow 0.
\Ee
This, together with \eqref{e:eAEst} and dominated convergence theorem, implies
that as $m \rightarrow \infty$,
\Bes
\begin{split}
\|I_3(t)\|_V &  \le C\int_0^t (t-s)^{-\frac 12} \|(I-\pi_m) N(X_s)\|_H ds \longrightarrow 0.
\end{split}
\Ees
It remains to estimate $I_4(t)$. By \eqref{e:eAEst}, \eqref{e:NHEst} and \eqref{e:UniBouXmX},
\Bes
\begin{split}
\|I_4(t)\|_V
& \le C\int_0^t (t-s)^{-\frac 12} \|N^m(X_s)-N^m(X^m_s)\|_H ds \\
& \le C \int_0^t (t-s)^{-\frac 12}\|N(X_s)-N(X^m_s)\|_H ds \\
& \le C \tl K \int_0^t (t-s)^{-\frac 12} \|X_s-X^m_s\|_V ds \\
& \le C \tl K \left(\int_0^t (t-s)^{-\frac p2} ds\right)^{\frac 1p} \left(\int_0^t \|X_s-X^m_s\|^q_V ds\right)^{\frac 1q},
\end{split}
\Ees
where $\tl K=\sup_{0 \le s \le T,m} (2+\|X_s\|^2_V+\|X^m_s\|^2_V) \le 2+2 \hat C^2$ and $\frac 1p+\frac 1q=1$ with $1<p<2$.
\vskip 3mm

Collecting the estimates for $I_1(t), ..., I_4(t)$, we get
\Bes
\begin{split}
\lim_{m \rightarrow \infty} \sup_m \|X_t-X^m_t\|_V
& \le  C \tl K t^{\frac 1p-\frac 12} \lim_{m \rightarrow \infty} \sup_m\left(\int_0^t \|X_s-X^m_s\|^q_V ds\right)^{\frac 1q} \\
& \le C \tl K t^{\frac 1p-\frac 12}  \left(\int_0^t \lim_{m \rightarrow \infty}  \sup_m\|X_s-X^m_s\|^q_V ds\right)^{\frac 1q}  \\
\end{split}
\Ees
where the last inequality is thanks to the fact $\sup_{0 \le s \le t} \|X_s-X^m_s\|_V \le 2\hat C$
and Fatou's theorem.
The above inequality implies
$$\lim_{m \rightarrow \infty} \sup_m \|X_t-X^m_t\|_V=0.$$
\end{proof}
\ \ \

Before proving Theorem \ref{t:Inv}, let us have a fast review about purely jump L\'evy processes as following.
Let $\{(l_j(t))_{t\geq 0},j\in \Z_*\}$ be a sequence of independent one dimensional purely jump L\'evy processes
with the same characteristic function, i.e.,
$$
\E e^{\mathrm{i}\xi l_j(t)}=e^{-t\psi(\xi)},\ \forall t\geq 0, j \in \Z_*,
$$
where $\psi(\xi)$ is a complex valued function called L\'evy symbol given by
$$
\psi(\xi)=\int_{\R \setminus\{0\}}(e^{\mathrm{i}\xi y}-1-\mathrm{i}\xi y1_{|y|\leq 1})\nu(dy),
$$
where $\nu$ is the L\'evy measure and satisfies that
$$
\int_{\R\setminus\{0\}}1\wedge |y|^2\nu(dy)<+\infty.
$$
For $t>0$ and $\Gamma\in \mcl B(\R\setminus\{0\})$,
the Poisson random measure associated with $l_j(t)$ is defined by
$$
N^{(j)}(t,\Gamma):=\sum_{s\in(0,t]}1_{\Gamma}(l_j(s)-l_j(s-)).
$$
The compensated Poisson random measure is given by
$$
\tilde N^{(j)}(t,\Gamma)=N^{(j)}(t,\Gamma)-t\nu(\Gamma).
$$
By L\'evy-It\^o's decomposition (cf. \cite[p.108, Theorem 2.4.16]{Ap04}), one has
$$
l_j(t)=\int_{|x|\leq 1}x\tilde N^{(j)}(t,dx)+\int_{|x|>1}x N^{(j)}(t,dx).
$$
\ \ \

\begin{proof} [Proof of Theorem \ref{t:Inv}] We follow the argument in \cite[(3.6)]{DXZ11}.
Write
$$
f(u):=(\|u\|^2_H+1)^{1/2},\ \ u \in H_m,
$$
it follows from It\^{o} formula (\cite{Ap04} or \cite{DXZ11}) that
$$f(X^m_t)=:f(x^m)-I^m_1(t)+I^m_2(t)+I^m_3(t)+I^m_4(t),$$
where
\begin{align*}
& I^m_1(t):=\int^t_0\frac{\|X^m_s\|^2_V}{(\|X^m_s\|^2_H+1)^{\frac 12}} ds+\int^t_0\frac{\langle X^m_s, N(X^m_s)\rangle_H}{(\|X^m_s\|^2_H+1)^{\frac 12}} ds, \\
& I^m_2(t):=\sum_{|j| \le m}\int^t_0\!\!\!\int_{|y|\leq 1}[f(X^m_s+y\beta_j e_j)-f(X^m_s)]\tilde N^{(j)}(ds,dy),\\
& I^m_3(t):=\sum_{|j| \le m}\int^t_0\!\!\!\int_{|y|\leq 1}\left[f(X^m_s+y\beta_j e_j)-f(X^m_s)
-\frac{\langle X^m_s,y\beta_je_j \rangle_0}{(\|X^m_s\|_H^2+1)^{\frac 12}}\right]\nu(dy) ds,\\
& I^m_4(t):=\sum_{|j| \le m}\int^t_0\!\!\!\int_{|y|> 1}\left[f_n(X^m_s+y\beta_j e_j)-f_n(X^m_s)\right]N^{(j)}(ds,dy).
\end{align*}
It follows from \eqref{e:NInnPro} that
\Bes
I^m_1(t) \ge \int^t_0\frac{\|X^m_s\|^2_V}{(\|X^m_s\|^2_H+1)^{\frac 12}} ds-\frac{t} 4.
\Ees
We apply the same argument as in \cite{DXZ11} to $I^m_2,...,I^m_4$ and get
\Bes
\begin{split}
& \E[\sup_{0 \le t \le T} |I^m_2(t)|] \le CT^{1/2}, \\
& \E[\sup_{0 \le t \le T} |I^m_3(t)|] \le CT, \\
& \E[\sup_{0 \le t \le T} |I^m_4(t)|] \le CT,
\end{split}
\Ees
where $C$ is some constant depending on $\alpha, \beta$ and $T>0$ is arbitrary.

Collecting the estimates about $I^m_1,...,I^m_4$, we immediately get
\Bes
\begin{split}
& \ \ \ \E\big[\sup_{t\in[0,T]}(\|X^m_t\|^2_H+1)^{\frac 12}\big]
+\E\int^T_0\frac{\|X^m_s\|^2_V}{(\|X^m_s\|^2_H+1)^{\frac12}}ds \\
& \leq (\|x\|^2_H+1)^{\frac 12}+CT+CT^{\frac 12}.
\end{split}
\Ees
It follows from Theorem \ref{t:GalCon} that for all $t>0$,
$$\lim_{m \rightarrow \infty} \|X^m_t\|_H=\|X_t\|_H \ \ \ a.s.,$$
$$\lim_{m \rightarrow \infty} \|X^m_t\|_V=\|X_t\|_V \ \ \ a.s..$$
By Fatou's Lemma, we have
\Bes
\begin{split}
& \ \ \ \E\big[\sup_{t\in[0,T]}(\|X_t\|^2_H+1)^{\frac 12}\big]
+\E\int^T_0\frac{\|X_s\|^2_V}{(\|X_s\|^2_H+1)^{\frac 12}}ds \leq (\|x\|^2_H+1)^{\frac 12}+CT+CT^{\frac 12}.
\end{split}
\Ees
This easily implies
\Bes
\begin{split}
\E\left(\int^T_0\|X_s\|_V ds\right)&\leq
\E\left(\int^T_0\frac{\|X_s\|_V(\|X_s\|_H+1)}{(\|X_s\|^2_H+1)^{\frac 12}} ds\right)\\
&\leq C \E\left(\int^T_0\frac{\|X_s\|^2_V+1}{(\|X_s\|^2_H+1)^{\frac 12}} ds\right)\\
&\leq C(1+\|x\|_H+T).
\end{split}
\Ees
This, together with the classical Bogoliubov-Krylov's argument, implies the existence of invariant measures
and that the support of invariant measures is $V$.
\end{proof}
\ \ \ \ \

\section{Strong Feller property}
For all $f \in B_b(H)$, define
$$P_t f(x)=\E[f(X^x_t)],$$
for all $t \ge 0$ and $x \in H$. By Theorem \ref{t:MaiThm}, $(P_t)_{t \ge 0}$ is a Markov semigroup on $B_b(H)$.
The main result of this section is
\begin{thm} \label{t:StrFelH}
$(P_t)_{t \ge 0}$, as a semigroup on $B_b(H)$, is strong Feller.
\end{thm}
\noindent To prove this theorem, we need to use the following theorem which will be proven later.
\begin{thm} \label{t:StrFel}
$(P_t)_{t \ge 0}$, as a semigroup on $B_b(V)$, is strong Feller.
\end{thm}
\ \ \ \

\begin{proof} [Proof of Theorem \ref{t:StrFelH}]
Let $T_0>0$ be arbitrary, it suffices to show that for all $t \in (0,T_0]$ and $x \in H$
\Bes
\lim_{\|y-x\|_H \rightarrow 0}P_t f(y)=P_t f(x).
\Ees
Define $\Omega_N:=\{\sup_{0 \le t \le T_0} \|Z(t)\|_{V} \le N\},$
it follows from Lemma \ref{l:ZEst} and Chebyshev's inequality that
\Be \label{e:OmeSma}
\PP(\Omega^c_N) \le c/N,
\Ee
where $c$ is some constants depending on $\alpha$ and $T_0$. \\

Let $x, y \in H$ be arbitrary and let $C>0$ be some constant
depending on $\|x\|_H, \|y\|_H$ and $N$, whose exact values may vary from line to line.

For $\omega \in \Omega_N$, denote by $Y.^x(\omega)$ and $Y.^y(\omega)$ the
solutions to Eq. \eqref{e:YEqn} with initial data $x$ and $y$ respectively. For the
notational simplicity, we shall omit the variable $\omega$ in functions if no confusions arise.

By (i) of Lemma \ref{l:LocExUnH},
there exists some constant $0<t_0 \le 1$, depending on $\|x\|_H, \|y\|_H$ and $N$,
such that for all $0<t\le t_0$
\Be \label{e:A14H}
\|A^{\frac 16} Y^x_t\|_{H} \le C t^{-1/6}, \ \ \ \|A^{\frac 16}Y^y_t\|_{H} \le C t^{-1/6}.
\Ee
Observe that
\Be \label{e:Xx-Xy}
X^x_t-X^y_t=I_1+I_2,
\Ee
where
\Bes
\begin{split}
I_1(t):=e^{-At} x-e^{-At} y, \ \ \ \ \  I_2(t):=\int_0^t e^{-A(t-s)}[N(X^x_s)-N(X^y_s)] ds. \\
\end{split}
\Ees
It follows from \eqref{e:eAEst} that
$$\|I_1(t)\|_V \le \tl c t^{-\frac 12} \|x-y\|_H,$$
where $\tl c$ is some constant.
Using \eqref{e:eAEst}, \eqref{e:NxyHEst} and \eqref{e:PoiInq}, we get
\Bes
\begin{split}
\|I_2(t)\|_V & \le \int_0^t \|A^{1/2}e^{-A(t-s)}\|\|N(X^x_s)-N(X^y_s)\|_H ds \\
& \le C \int_0^t (t-s)^{-\frac 12} (1+\|A^{\frac 16}X^x_s\|_H^2+\|A^{\frac 16}X^y_s\|_H^2)\|X^x_s-X^y_s\|_V ds.
\end{split}
\Ees
By \eqref{e:A14H},
\Bes
\|A^{\frac 16}X^x_s\|_H \le \|A^{\frac 16}Y^x_s\|_H+\|A^{\frac 16}Z_s\|_H \le C(s^{-1/6}+1)
\Ees
Similarly, $\|A^{\frac 16}X^y_s\|_H \le C(s^{-1/6}+1)$. Since $0<s \le t \le t_0 \le 1$, we further get
\Bes
1+\|A^{\frac 16}X^x_s\|_H^2+\|A^{\frac 16}X^y_s\|_H^2\le Cs^{-\frac 13}.
\Ees
Hence,
\Bes
\|I_2(t)\|_V \le  C\int_0^t (t-s)^{-\frac 12} s^{-\frac 13} \|X^x_s-X^y_s\|_V ds.
\Ees
\vskip 3mm

For any $r \in (0,t_0]$, define
$$\Phi_r=\sup_{0 \le t \le r} t^{\frac 12} \|X^x_t-X^y_t\|_V,$$
by (i) of Lemma \ref{l:LocExUnH},
$$\Phi_r \le \sup_{0 \le t \le r} t^{\frac 12} (\|Y^x_t\|_V+\|Y^y_t\|_V)+2r^{\frac 12} N<\infty.$$
It follows from \eqref{e:Xx-Xy} and the bounds of $I_1$, $I_2$ that
\Bes
\begin{split}
\Phi_r & \le \tl c\|x-y\|_H+C\sup_{0 \le t \le r}[t^{\frac 12}\int_0^t (t-s)^{-\frac 12} s^{-\frac 56} ds] \Phi_r \\
& \le \tl c\|x-y\|_H+Cr^{\frac16} \Phi_r.
\end{split}
\Ees
Choose $r$ so small that $Cr^{\frac 16} \le \frac 12$, we get
\Bes
\Phi_r  \le 2\tl c \|x-y\|_H,
\Ees
this immediately implies
\Be \label{e:UFroH}
\|X^x_t-X^y_t\|_V \le 2 \tl c t^{-\frac 12} \|x-y\|_H, \ \ \ \ 0<t \le r.
\Ee
\vskip 3mm

By the Markov property, for all $0<t \le T_0$,
we have
\Bes
\begin{split}
|P_t f(x)-P_t f(y)|=|\E[P_{t-s} f(X^x_s)-P_{t-s} f(X^y_s)]| \le J_1+J_2
\end{split}
\Ees
where $s=\frac t2 \wedge r$ and
$$J_1:=|\E\left\{[P_{t-s} f(X^x_s)-P_{t-s} f(X^y_s)] 1_{\Omega^c_N}\right\}|,$$
$$J_2:=|\E\left\{[P_{t-s} f(X^x_s)-P_{t-s} f(X^y_s)] 1_{\Omega_N}\right\}|.$$
By \eqref{e:OmeSma},
$$J_1 \le 2 c\|f\|_\infty/N.$$
It follows from \eqref{e:UFroH}, Theorem \ref{t:StrFel} and dominated convergence theorem that
$$J_2 \rightarrow 0, \ \ \ \ \  \ \ \ \|x-y\|_H \rightarrow 0.$$
Combining the estimates of $J_1$ and $J_2$, we immediately conclude the proof.
\end{proof}
\ \ \

Let us now discuss the method of proving Theorem \ref{t:StrFel}.
To show the strong Feller property of the semigroup $(P_t)_{t \ge 0}$ on some function space $B_b(W)$,
the noise $(L_t)_{t \ge 0}$ under the norm $\|.\|_{W}$ need to be sufficiently strong to get
a gradient estimate for the OU semigroup corresponding to $(Z_t)_{t \ge 0}$.
If $W=H$, $(L_t)_{t \ge 0}$ is not strong enough. Therefore, we choose
$W=V$ to make the norm of $L_t$ larger.
\vskip 3mm

Because the nonlinearity $N$ is not bounded, we need to use a
well known truncation technique, i.e., considering the equation
with truncated nonlinearity as follows:
\Be \label{e:TruEqn}
dX^{\rho}_t+[A X^{\rho}_t+N^{\rho}(X^{\rho}_t)]dt=d L_t, \ \ \ X^\rho_0=x \in V.
\Ee
where $\rho>0$, $N^{\rho}(x)=N(x) \chi(\frac{\|x\|_V}\rho)$ for all $x \in V$ and
$\chi: \R \rightarrow [0,1]$ is a smooth function such that
$$\chi(z)=1 \ \ \ {\rm for} \ |z| \le 1, \ \ \ \ \chi(z)=0 \ \ \ {\rm for} \ |z| \ge 2.$$
By \eqref{e:NVEst}, for all $x \in V$,
\Be \label{e:BouB}
\|N^{\rho} (x)\|_V \le C (\|x\|_V^3+\|x\|_V) \chi(\frac{\|x\|_V}{\rho}) \le C(\rho^3+\rho).
\Ee
One can easily check that $N^{\rho}$ is a Lipschitz function from $V$ to $V$. Hence,
Eq. \eqref{e:TruEqn} admits a unique solution $X^\rho_{.} \in D([0,\infty);V)$.
For every $f \in B_b(V)$, define
$$P^\rho_t f(x)=\E[f(X^{\rho,x}_t)], \ \ \ \ t \ge 0, \ x \in V,$$
$(P^\rho_t)_{t \ge 0}$ is a Markov semigroup.
\vskip 2mm

To establish the gradient estimate of $(P^\rho_t)_{t \ge 0}$, let us first define the derivative of $f \in C^1_b(V)$: given an $h \in V$,
\Bes
D_h f(x):=\lim_{\e \rightarrow 0} \frac{f(x+\e h)-f(x)}{\e}.
\Ees
By Riesz representation theorem, for every $x \in V$, there exists some $Df(x) \in V$ such that
$$D_h f(x)=\Ll Df(x), h\Rr_V, \ \ \ \ \ \ h \in V.$$
We define
\Be \label{e:DfV}
\|Df\|_\infty=\sup_{x \in V} \|Df(x)\|_V.
\Ee

\vskip 3mm

\begin{prop} \label{p:DPTru}
Let $f \in B_b(V)$.
For all $\alpha \in (3/2,2)$, there exists some $\theta \in [1/\alpha, 1)$ such that
\Be
\|DP^{\rho}_t f\|_\infty \le C t^{-\theta}\|f\|_\infty, \ \ \ \ t>0,
\Ee
where $C>0$ depends on $\rho$, $\alpha$ and $\theta$.
\end{prop}

\begin{proof}
Observe that $L_t=\sum_{k \in \Z_{*}} \beta_k l_k(t) e_k$ is represented in the space $V$ by
\Bes
L_t=\sum_{k \in \Z_{*}} \tl \beta_k l_k(t) \tl e_k,
\Ees
where $\tl \beta_k=\gamma_k^{1/2} \beta_k$, $\tl e_k=\gamma^{-1/2}_k e_k$
for $k \in \Z_{*}$. $\{\tl e_k\}_{k \in \Z_{*}}$ is an
orthonormal basis of $V$.
Recall the condition in (ii) of Eq. \eqref{e:XEqn}: there exists some $C_1, C_2>0$ such that
\Be \label{e:BeKCon}
C_1 \gamma^{-\beta}_k \le |\beta_k| \le C_2 \gamma^{-\beta}_k,
\Ee
it is easy to check that as
\Bes \label{e:BetSel}
\beta<\frac 32-\frac 1{\alpha},
\Ees
we have
\Be \label{e:ConPZ}
|\tl \beta_k| \ge C\gamma_k^{-(\theta-\frac1{\alpha})},
\Ee
where $\theta \in [1/\alpha,1)$ and $C>0$ is some constant.
Note that \eqref{e:ConPZ} is
Hypothesis (5.2) of \cite{PZ11}, so we get (5.19) of \cite{PZ11}, i.e., there exists some constant
$C>0$ depending on $\alpha$ such that
\Bes
\|D R_t f\|_\infty \le Ct^{-\theta} \|f\|_\infty, \ \ \ \ \ \ f \in B_b(V).
\Ees
where $(R_t)_{t \ge 0}$ is the semigroup corresponding to the OU process $(Z_t)_{t \ge 0}$.
\vskip 3mm

To make \eqref{e:BeKCon} and \eqref{e:ConPZ} be both satisfied, we need
\Be \label{e:BetSel}
\frac 12+\frac 1{2\alpha}<\beta<\frac 32-\frac 1{\alpha}.
\Ee
To make the condition \eqref{e:BetSel} be
satisfied, we need
$$\alpha \in (3/2,2).$$
Recall that $N^\rho$ is a bounded Lipschitz function (see \eqref{e:BouB}), by Lemma 5.9 of \cite{PZ11}, we have
\Bes
P^{\rho}_t f(x)=R_t f(x)+\int_0^t R_{t-s}[\Ll N^{\rho}, DP_s^{\rho} f\Rr_V](x)ds
\Ees
and the desired inequality.
\end{proof}

Define
\Be \label{e:TauX}
\tau_{x}:=\inf\{t>0; \|X^{x}_t\|_V \ge \rho\},
\Ee
by (\ref{3}) of Theorem \ref{t:MaiThm}, $\tau_x$ is a stopping time. For all $t<\tau_x$, Eqs. \eqref{e:XEqn}
and \eqref{e:TruEqn} have the same solutions. Thanks to the following two points: one is the semigroup $(P^\rho_t)_{t \ge 0}$ has a gradient estimate, the other is
the stopping time can be estimated, we can prove the strong Feller property of
$(P_t)_{t \ge 0}$.
\vskip 3mm

\begin{proof} [Proof of Theorem \ref{t:StrFel}]
Without loss of generality, we
assume $\|f\|_\infty=1$. Let $T_0>0$ be arbitrary, it suffices to show that for all $t \in (0,T_0]$ and $x \in V$
\Be \label{e:PtfyxLim}
\lim_{\|y-x\|_V \rightarrow 0}P_t f(y)=P_t f(x).
\Ee

Recall
$$K_{T_0}(\omega):= \sup_{0 \le t \le T_0}\|Z_t(\omega)\|_{V}, \ \ \ \omega \in \Omega,$$
by Lemma \ref{l:ZEst} and Markov inequality we have
\Be \label{e:KOmeSma}
\PP(K_{T_0}>\rho/2) \le  \frac{C}{\rho},
\Ee
where $C$ is some constant depending on $\alpha$ and $T_0$.

Choose $\rho$ so large that $\|x\|_V\le \sqrt \rho$ and define
$$A:=\{K_{T_0}\le \rho/2\},$$
By (ii) of Lemma \ref{l:LocExUnH}, there exists some
$0<t_0 \le T_0$ depending on $\rho$ such that for all
$\omega \in A$,
\Be \label{e:YtOmeSma}
\sup_{0 \le t \le t_0} \|Y^x_t(\omega)\|_V \le 1+\|x\|_V.
\Ee
Observe that
\Be \label{XtOmeSma}
\begin{split}
& \ \ \PP(\sup_{0 \le t \le t_0} \|X^x_t\|_V \ge \rho) \le \PP(\sup_{0 \le t \le t_0} \|Y^x_t\|_V
+\sup_{0 \le t \le T_0} \|Z_t\|_V \ge \rho) \\
& \le \PP(K_{T_0}>\rho/2)+\PP(\sup_{0 \le t \le t_0} \|Y^x_t\|_V>\rho/2, A)
\end{split}
\Ee
By \eqref{e:YtOmeSma}, for all $\omega \in A$, we have
\Bes
\sup_{0 \le t \le t_0} \|Y^x_t(\omega)\|_V \le 1+\|x\|_V  \le 1+\sqrt \rho<\rho/2.
\Ees
So,
\Be  \label{e:YtOme0}
\PP(\sup_{0 \le t \le t_0} \|Y^x_t\|_V>\rho/2, A)=0.
\Ee
Hence,
\Be
\PP(\sup_{0 \le t \le t_0} \|X^x_t\|_V \ge \rho) \le \PP(K_{T_0}>\rho/2)\le C/\rho,
\Ee
where the last inequality is by \eqref{e:KOmeSma}. It follows from the above inequality
that for all $t \in [0,t_0]$
\Be \label{e:TauEst}
\PP_x(\tau_{x} \le t)=\PP(\sup_{0 \le s \le t} \|X^x_s\|_V \ge \rho) \le  C/\rho.
\Ee
Since Eq. \eqref{e:XEqn} and Eq. \eqref{e:TruEqn} both have a unique mild solution, for all $t \in [0, \tau_{x})$,
we have
\Be \label{e:StrWeaUni}
X^{\rho,x}_t=X^x_t \ \ a.s..
\Ee
\ \ \

Let $y \in V$ be such that $\|x-y\|_V \le 1$ and choose $\rho>0$ be sufficiently
large so that $\|x\|_V, \|y\|_V  \le \sqrt{\rho}$.
Let $t \in (0,t_0]$, observe
\Be
\begin{split}
|P_t f(x)-P_t f(y)|=|\E [f(X^x_t)]-\E [f(X^y_t)]|=I_1+I_2+I_3,
\end{split}
\Ee
where
\Bes
\begin{split}
& I_1:=|\E[f(X^x_t)1_{\tau_{x}> t}]-\E[f(X^y_t)1_{\tau_{y} > t}]|, \\
& I_2:=|\E[f(X^x_t)1_{\tau_{x} \le t}]|, \ \ I_3:=|\E[f(X^y_t)1_{\tau_{y} \le t}]|.
\end{split}
\Ees
It follows from \eqref{e:TauEst} that
$$I_2 \le \frac{C}{\rho}, \ \ I_3 \le \frac{C}{\rho}.$$
\ \ \ \

\noindent It remains to estimate $I_1$. It follows from \eqref{e:StrWeaUni}, Proposition
\ref{p:DPTru} and \eqref{e:TauEst} that
\Bes
\begin{split}
I_1& =|\E[f(X^{\rho,x}_t)1_{\tau_{x}>t}]-\E[f(X^{\rho,y}_t)1_{\tau_{y} > t}]| \\
&\le |\E[f(X^{\rho,x}_t)]-\E[f(X^{\rho,y}_t)]|+|\E[f(X^{\rho,x}_t)1_{\tau_{x} \le t}]|+|\E[f(X^{\rho,y}_t)1_{\tau_{y} \le t}]|\\
& \le \tl Ct^{-\theta} \|x-y\|_V+2C/\rho
\end{split}
\Ees
where $\tl C$ depends on $\alpha, \rho$ and $\theta$. For all $\e>0$,
choosing $$\rho \ge \max\{\frac{12C}{\e}, 2\|x\|_V^2+2\}, \ \ \delta<\frac{\e t^{\theta}}{2 \tl C},$$
as $\|x-y\|_V \le \delta$, we have
$$|P_t f(x)-P_t f(y)|<\e, \ \ \  \ \ t \in (0,t_0].$$

As $t_0<t \le T_0$, it follows from Markov property and the strong Feller property above that
\Bes
P_t f(y)-P_t f(x)=P_{t_0} [P_{t-t_0} f](y)-P_{t_0} [P_{t-t_0} f](x) \rightarrow 0
\Ees
as $\|y-x\|_V \rightarrow 0$.
\end{proof}
\ \ \ \ \

\section{Proof of Theorem \ref{t:Erg}}
Here we can not use the classical Doob's Theorem to get the ergodicity because
we are not able to prove the irreducibility. Alternatively, we shall use a simple but useful criterion in \cite{Hai09}.
Let us first introduce the conception of accessibility.

\begin{defn} [Accessibility]
Let $(X_t)_{t \ge 0}$ be a stochastic process valued on a metric space $E$ and let $(P_t(x,.))_{x \in E}$ be the transition
probability family. $(X_t)_{t \ge 0}$ is said to be
accessible to $x_0 \in E$ if the resolvent $\mcl R_{\lambda}$ satisfies
$$\mcl R_{\lambda}(x, \mcl U):=\int_0^\infty e^{-\lambda t}P_t(x, \mcl U) dt>0$$
for all $x \in E$ and all neighborhoods $\mcl U$ of $x_0$, where $\lambda>0$ is arbitrary.
\end{defn}

The simple but useful criterion we shall use is the following theorem
\begin{thm}[Corollary 7.8,\cite{Hai09}] \label{t:SimCri}
If $(X_t)_{t \ge 0}$ is strong Feller at an accessible point $x \in E$, then it can
have at most one invariant measure.
\end{thm}

\begin{proof} [Proof of Theorem \ref{t:Erg}]
For all $\e>0$ and $t>0$, define
$\Omega_{\e,t}=\{\sup_{0 \le s \le t} \|Z_s\|_V \le \e\},$
it follows from Lemma \ref{l:Acc} that
\Bes
\PP(\Omega_{\e,t})>0.
\Ees
Recall \eqref{e:PriEst}, for all $\omega \in \Omega_{\e,t}$ we get
\Bes
\begin{split}
\|Y_t(\omega)\|_H^2 & \le  e^{-(2 \pi-3)t}\|x\|^2_H+\int_0^t e^{-(2 \pi-3)(t-s)} (\|Z_s(\omega)\|^2_H+C\|Z_s(\omega)\|_V^4) ds \\
& \le e^{-(2 \pi-3)t}\|x\|^2_H+C(\e^2+\e^4).
\end{split}
\Ees
For all $r>0$, denote
$$B_H(r):=\{x \in H; \|x\|_H<r\}.$$ For all $R>0$,
it follows from the previous inequality that for all $\delta>0$, we can choose $T:=T_{R,\delta}$ sufficiently large and $\e:=\e_{R,\delta}$ sufficiently
small so that, as $t \ge T$, for all $x \in B_H(R)$ and $\omega \in \Omega_{\e,t}$,
\Bes
\begin{split}
\|X^x_t(\omega)\|_H & \le \|Y^x_t(\omega)\|_H+\|Z_t(\omega)\|_H \\
& \le e^{-(\pi-\frac 32)t} R+C(\e^4+\e^2+\e)<\delta.
\end{split}
\Ees
Since $\PP(\Omega_{\e,t})>0$, we have for all $x \in B_H(R)$
\Be \label{e:Xxt<Del}
P(t;x,B_H(\delta))>0, \ \ \ \ \ \ \ t\ge T.
\Ee
This clearly implies for all $x \in B_H(R)$ and $\lambda>0$,
$$\mcl R_{\lambda}(x,B_H(\delta))>0.$$
Since $R>0$ is arbitrary, the above inequality is true for all
$x \in H$ and thus $(X_t)_{t \ge 0}$ is accessible to $0$.
\\

Of course, we can apply Theorem \ref{t:SimCri} to get
the ergodicity immediately. However, following the spirit in \cite{Hai09}, we can give a clear and short proof
as follows. \\

If $\mu$ is an invariant measure, it follows from Theorem \ref{t:Inv} that $\mu$ is supported on
$V$. Therefore, there exists some (large) $R>0$ such that
\Be \label{e:MuBR>0}
\mu(B_H(R))>0.
\Ee
The inequalities \eqref{e:Xxt<Del} and \eqref{e:MuBR>0} immediately imply
\Be
\mu(B_H(\delta))=\int_H P(T;x,B_H(\delta)) \mu(dx)>0, \ \ \ \ \ \ \ \forall \ \delta>0.
\Ee
Assume that Eq. \eqref{e:XEqn} admits two invariant measures $\mu_1$ and $\mu_2$.
It is well known (\cite{Hai09}) that there are two sets $A_1$ and $A_2$ such that
$\mu_1(A_1)=1$, $\mu_2(A_2)=1$ and $A_1 \cap A_2=\emptyset$. Observe that
\Be \label{e:Mu1}
\int_{H} P(t;x,A_1) \mu_1(dx)=\mu_1(A_1)=1,
\Ee
\Be \label{e:Mu2}
\int_{H} P(t;x,A_1) \mu_2(dx)=\mu_2(A_1)=0.
\Ee
It follows from \eqref{e:MuBR>0} that $\mu_1(B_H(\delta))>0$ for all $\delta>0$. By the strong Feller property (Theorem
\ref{t:StrFelH}) and \eqref{e:Mu1}, we have
$$P(t;0,A_1)=1.$$
On the other hand, by strong Feller property and $\mu_2(B_H(\delta))>0$ again,
\Bes
\int_{H} P(t;x,A_1) \mu_2(dx) \ge \int_{B_H(\delta)} P(t;x,A_1) \mu_2(dx)>0.
\Ees
This is contradictory to \eqref{e:Mu2}. So Eq. \eqref{e:XEqn} admits a unique invariant measure.
\end{proof}

\section{Appendix: Some estimates about the nonlinearity $N$.}
Let us show \eqref{e:NInnPro}-\eqref{e:NHEst}. It follows from Young's inequality that
\Bes
\langle x, -N(x)\rangle_H=\langle x, x-x^3\rangle_H=\int_{\T} |x(\xi)|^2 d\xi-\int_\T |x(\xi)|^4 d\xi \le \int_{\T}\frac 14 d\xi \le \frac 14.
\Ees
By Sobolev embedding theorem and \eqref{e:PoiInq}, we have
\Bes
\begin{split}
\|N(x)\|_V^2 & \le C\int_{\T} |\p_\xi x(\xi)|^2 d\xi+C\int_{\T} |x(\xi)|^4 |\p_\xi x(\xi)|^2 d\xi \\
& \le C \|x\|^2_V+C\|x\|_\infty^4 \|x\|^2_V \\
& \le C \|x\|^2_V+C\|x\|_\infty^4 \|x\|^2_V \\
& \le C \|x\|^2_V+C\|A^{\frac 14} x\|_H^4 \|x\|^2_V \\
& \le C \|x\|^2_V+C\|x\|^6_V.
\end{split}
\Ees
For \eqref{e:NxyHEst}, it follows from Sobolev embedding theorem and Young's inequality that
\Bes
\begin{split}
\|N(x)-N(y)\|_H&=\|x-y\|_H+\|(x-y)(x^2+y^2+xy)\|_H \\
& \le \|x-y\|_H+C(\|x\|^2_{L^6}+\|y\|^2_{L^6}) \|x-y\|_{L^6} \\
& \le C\|A^{\frac 16}(x-y)\|_H+C(\|A^{\frac 16}x\|^2_H+\|A^{\frac 16}y\|^2_H)\|A^{\frac 16}(x-y)\|_H  \\
& \le C(1+\|A^{\frac 16}x\|^2_H+\|A^{\frac 16}y\|^2_H)\|A^{\frac 16}(x-y)\|_H \\
& \le C (1+\|A^{\sigma}x\|^2_H+\|A^{\sigma}y\|^2_H)\|A^{\sigma}(x-y)\|_H,
\end{split}
\Ees
where the last inequality is by \eqref{e:PoiInq}.
Let $y=0$ and apply Young's inequality, we immediately get \eqref{e:NHEst} from \eqref{e:NxyHEst}.

It remains to show \eqref{e:NxyHEst1}. By Sobolev embedding theorem again, we have
\Bes
\begin{split}
\|N(x)-N(y)\|_H&=\|x-y\|_H+\|(x-y)(x^2+y^2+xy)\|_H \\
& \le \|x-y\|_H+C(\|x\|^2_{L^\infty}+\|y\|^2_{L^\infty}) \|x-y\|_{H} \\
& \le C(1+\|A^{\frac 14} x\|_H^2+\|A^{\frac 14} y\|_H^2) \|x-y\|_{H}.
\end{split}
\Ees

\bibliographystyle{amsplain}

\end{document}